\documentclass[]{amsart}

\def\bn{\mathbb{N}}
\def\br{\mathbb{R}} 
\def\bc{\mathbb{C}} 
\def\h{\mathcal{H}} 
\def\k{\mathcal{K}}

\newcommand\ca{\mathcal{A}}

\newcommand{\bh}{{\rm B}(\mathcal{H})}
\newcommand{\bk}{{\rm B}(\mathcal{K})}

\newcommand{\bhk}{{\rm B}(\mathcal{H},\mathcal{K})}

\newcommand{\matn}[1]{{\rm M}_n(#1)}

\newcommand{\mat}[2]{{\rm M}_{#1}(#2)}

\newcommand{\rea}[1]{{\rm Re}\,{#1}}
\newcommand{\po}{\perp_{P}}

\newtheorem{theorem}{Theorem}[section]
\newtheorem{lemma}[theorem]{Lemma}
\newtheorem{co}[theorem]{Corollary}
\newtheorem{pr}[theorem]{Proposition}
\theoremstyle{remark}
\newtheorem{re}[theorem]{Remark}
\theoremstyle{definition}
\newtheorem{de}[theorem]{Definition}
\newtheorem{ex}[theorem]{Example}
\numberwithin{equation}{section}

\makeatletter
\@namedef{subjclassname@2020}{\textup{2020} Mathematics Subject Classification}
\makeatother

\begin{document}

\title[]{Orthonormal pairs of operators} 

\author{Bojan Magajna} 
\address{Department of Mathematics\\ University of Ljubljana\\
Jadranska 21\\ Ljubljana 1000\\ Slovenia}
\email{Bojan.Magajna@fmf.uni-lj.si}

\thanks{Acknowledgment. I am grateful to the anonymous Referee for his remarks and for bringing two references to my attention.}

\thanks{The author acknowledges the financial support from the
Slovenian Research Agency (research core funding no. P1-0288).}



\begin{abstract}We consider pairs of operators $A,B\in\bh$, where $\h$ is a Hilbert space, such  that there exist a linear isometry $f$ from the span of $\{A,B\}$ into $\bc^2$ mapping $A,B$ into orthonormal vectors. We prove some necessary conditions for the existence of such an $f$ and  determine all such pairs among commuting normal operators. Then we characterize all such pairs $A,B$ (in fact,  we consider general sets instead of just pairs) under the additional requirement that $f$ is a complete isometry, when $\h$ carries the column (or row) operator space structure.  We also metrically characterize  elements in a C$^*$-algebra  with orthogonal ranges. 
\end{abstract}

\maketitle

\section{Introduction}

Various notions of orthogonality for vectors in a Banach space were introduced already by Birkhoff \cite{B} and James \cite{J} (recent surveys by Bottazzi, Conde and Sain and by Grover and Sushil  are \cite{BCS} and \cite{GS}), which  were investigated  even in the context of Hilbert C$^*$-modules by Aramba\v si\'c and Raji\' c \cite{AR}.   One possible natural definition of orthogonality, investigated by Eskandari, Moslehian and Popovici in \cite{EMP}  and called {\em Pythagoras orthogonality}, is the following: two vectors $x,y$ in a normed space are orthogonal, which is denoted as $x\po y$, if there exists a linear isometry $f$ from the linear span of $x$ and $y$ into a Hilbert space such that the vectors $f(x)$ and $f(y)$ are orthogonal in the usual sense. This is a much more restricted notion than Pythagorean orthogonality and other types of orthogonalitiy introduced in \cite{J}. 

For the Banach space $\bh$ of all bounded linear operators on a Hilbert space $\h$, perhaps the most studied kind of orthogonality, the Birkhoff-James orthogonality, has been characterized in terms of numerical ranges \cite{St}, \cite[2.2]{M} (although the term ``Birkhoff-James orthogonality'' was not used there) and by Bhatia and \v Semrl \cite{BS}.  Here we will study Pythagoras orthogonality in $\bh$, especially in the case when $\h$ is finite-dimensional, and an appropriate variant of it in the context of operator spaces. 

In Section 2 we will present some necessary conditions for Pythagoras orthogonality. For example, for each such pair $A,B\in\matn{\bc}$ and each $\lambda\in\bc$ the operator $A+\lambda B$ is not invertible. This is in sharp contrast with the situation in Hilbert spaces over $\br$, where the Pauli spin matrices are an example of mutually orthogonal hermitean unitary operators. In Section 3 we characterize Pythagoras orthogonality for commuting normal operators in terms of their joint spectrum.  Then, following the paradigm that the  natural maps between operator spaces are the completely bounded maps (instead of all bounded ones\cite{BL}, \cite{ER}, \cite{Pa}, \cite{Pi}), we  consider the central topic of this article. Namely,  the question, which sets of operators in $\bk$ can be mapped completely isometrically into orthonormal sets in a Hilbert space $\h$, where $\h$ carries the column operator space structure as defined e. g. in \cite{ER}, \cite{Pi}. We call such sets of operators {\em column orthonormal}.  It turns out in Section 4 that, in contrast to the usual Pythagoras orthogonality, column orthogonal sets of operators have a more definitive simple characterization.  A set of norm $1$ operators $C_j\in\bh$ is shown to be column orthogonal if and only if $\|\sum_jC_jC_j^*\|\leq 1$ and there exists a representation $\pi$ of $\bh$ on some Hilbert space $\k$ and a cyclic vector $\xi\in\k$ for $\pi$ such that $\|\pi(C_j)\xi\|=1$ simultaneously for all $j$. In Section 5 we study what happens if in the condition $\|A\alpha+B\beta\|^2=\|A\|^2|\alpha|^2+\|B\|^2|\beta|^2$ for Pythagoras orthogonality we replace scalars $\alpha$ and $\beta$ with elements $X,Y$of a  C$^*$-algebra $\ca$ and replace $\|A\|^2$, $\|B\|^2$ with positive elements $P,Q\in\ca$. That is, for fixed positive $P,Q\in\ca$ we characterize  all pairs of elements $A,B\in\ca$ that satisfy the identity
$$\|X^*A^*AX+Y^*B^*BY\|=\|X^*PX+Y^*QY\|$$
for all $X,Y\in\ca$. We show that in this case $A$ and $B$ have orthogonal ranges (that is, $A^*B=0$) and, in the special case when $Q=P$ is a projection, $A$ and $B$ are necessarily partial isometries with the same initial projection $P$, so that $A$ and $B$ are column orthonormal.
Finally, in Section 6 we return to Pythagoras orthogonality and describe all operators $A$ that are Pythagoras orthogonal to a projection of rank $1$. This indicates that such orthogonality depends on the action of $A$ on the entire Hilbert space.

Note that if $A\po B$ for operators $A,B\in\bh$, then also $f(A)\po f(B)$ for each linear or conjugate-linear isometry $f$ of $\bh$. Particular examples of linear isometries on $\bh$ are: (i) maps of the form $X\mapsto UXV$, where $U,V^*\in\bh$ are isometries, and (ii) $X\mapsto X^t$. (All linear surjective isometries are known to be compositions of this to types \cite[10.5.26, 10.5.32]{KR}, \cite{So}.) The map $X\mapsto X^*$ is a conjugate linear isometry. Thus, $A\po B$ if and only if $A^*\po B^*$.
Recall that for $B\in\bh$ at least one of the operators $B$, $B^*$ has a polar decomposition in which the partial isometric part is an isometry. Suppose that $B=U|B|$, where $U^*$ is an isometry. Then $A\po B$ implies that $U^*A\po|B|$. The reverse implication also holds, since $U$ is isometric on the range of $U^*$, which contains  the ranges of $U^*A$ and $|B|$. Also, multiplying $A$ and $B$ by nonzero scalars does not change orthogonality.
Therefore, for shorter formulation of results  we will often assume that 

{\em $A$ and $B$ are linear operators on a Hilbert space $\h$ with $\|A\|=1=\|B\|$.}

\section{Some necessary conditions for Pythagoras orthogonality}

By definition, for $A, B\in\bhk$ with $\|A\|=1=\|B\|$, the condition $A\po B$ means that $\|A+\lambda B\|^2=1+|\lambda|^2$. This can be written as
$$\|(A+\lambda B)^*(A+\lambda B)\|=1+|\lambda|^2.$$
Since for a positive operator $T$ with the spectrum $\sigma(T)$ the relations $\|T\|\in\sigma(T)$ and $T\leq\|T\|I$ hold, where $I$ is the identity operator,  we can reformulate the orthogonality condition in the following way, stated as a lemma for easier reference:

\begin{lemma}\label{le1} Suppose that $\|A\|=1=\|B\|$, where $A,B\in\bhk$. Then $A\po B$ if and only if for all $\lambda\in\bc$ the operator $$(1+|\lambda|^2)I-(A+\lambda B)^*(A+\lambda B)$$
is positive and not invertible.
\end{lemma}

\begin{pr}\label{pr2}Suppose that  $A,B\in\mat{n}{\bc}$ and $\|A\|=1=\|B\|$. If $A\po B$ then for each $\lambda\in\bc$ the operator $A+\lambda B$ is not invertible. 
\end{pr}

\begin{proof}By Lemma \ref{le1} we have 
$$\det[(|\lambda|^2+1)I-(A+\lambda B)^*(A+\lambda B)]=0$$
for each $\lambda\in\bc$.  For $\lambda\in\br$ this can be rewritten as
$$\det[\lambda^2(I-B^*B))-\lambda(B^*A+A^*B)+I-A^*A]=0.$$
Since the left side of this equality is a polynomial in $\lambda$, it must be identically $0$ for all $\lambda\in\bc$. In particular for $\lambda= i$ we get
$\det[B^*B-A^*A- i(B^*A+A^*B)]=0,$ which we may rewrite as
$$\det[(B^*-iA^*)(B-iA)]=0.$$
We may replace in this argument $A$ by $\omega A$ for any $\omega\in\bc$ with $|\omega|=1$, since $\omega A\po B$ and $\|\omega A\|=1$. Thus $\det(B+i\omega A)=\overline{\det(B^*-i\overline{\omega} A^*)}=0$ or $\det(B-i\omega A]=0$, and at least one of these two possibilities holds for infinitely many values of $\omega$.
Since these determinants are polynomials in $\omega$, it follows that at least one of them is identically $0$ for all $\omega\in\bc$ and this clearly implies that $A+\lambda B$ is not invertible for $\lambda\in\bc$.
\end{proof}

\begin{re}Since Pythagoras orthogonality is symmetric relation, it follows that in  Proposition \ref{pr2} also the operator $B$ is not invertible. 
\end{re}

{\bf Problem.} Can Proposition \ref{pr2} be generalized to operators on infinite dimensional Hilbert spaces?

\smallskip
Suppose that there exists  a unit vector $\xi\in\h$ so that $\|B\xi\|=\|B\|=1$. If $B\geq0$, this means that $B\xi=\xi$ (since $0\leq\|(I-B)\xi\|^2=2-2\|\sqrt{B}\xi\|^2\leq0$, where the inequality follows from $1=\|\sqrt{B}\sqrt{B}\xi\|\leq\|\sqrt{B}\xi\|$). If $A\po B$, then from the inequality
$$\|A\xi\|^2+2\rea(\lambda\langle \xi,A\xi\rangle)+|\lambda|^2=\|(A+\lambda B)\xi\|^2\leq\|A+\lambda B\|^2=1+|\lambda|^2,$$
which can be written as $2\rea(\lambda\langle\xi,A\xi\rangle)\leq1-\|A\xi\|^2,$
we conclude (by considering $|\lambda|\to\infty$) the following lemma:

\begin{lemma}\label{le3}Suppose that $\|A\|=1=\|B\|$, $B\geq0$ and that $A\po B$. If $\xi\in\h$ is such that $\|\xi\|=1=\|B\xi\|$, then $\langle A\xi,\xi\rangle=0.$
\end{lemma}

Now we can easily classify all pairs of orthogonal operators $A,B$ in $\mat{2}{\bc}$.

\begin{pr}\label{pr4}The only pair $A,B\in\mat{2}{\bc}$ with $\|A\|=1=\|B\|$ that satisfies $A\po B$ is, up to an isometry of $\mat{2}{\bc}$,
\begin{equation}\label{1}A=\left[\begin{array}{ll}
0&0\\
1&0\end{array}\right],\ \ \ B=\left[\begin{array}{ll}
1&0\\
0&0\end{array}\right].\end{equation}
\end{pr}

\begin{proof}We may assume $B\geq0$. Since $B$ is not invertible by Proposition \ref{pr2}, one of the eigenvalues of $B$ is $0$, hence we may assume that $B$ is as stated in (\ref{1}). Then it follows from Lemma \ref{le3} that $A$ is of the form
$$A=\left[\begin{array}{ll}
0&\beta\\
\gamma&\delta\end{array}\right]\ \ \ (\beta,\gamma,\delta\in\bc).$$
But by Proposition \ref{pr2} $A+\lambda B$ is not invertible, hence $0\equiv\det(A+\lambda B)=\delta\lambda-\beta\gamma$. 
Thus $\delta=0$ and $\beta\gamma=0$. If $\gamma=0$, we apply the transposition, so in any case we can achieve that $A$ is of the form
$$A=\left[\begin{array}{ll}
0&0\\
\gamma&0\end{array}\right],\ \ \mbox{where}\ |\gamma|=1.$$ Multiplying $A$ and $B$ from the left by the unitary matrix $U=\left[\begin{array}{ll}
1&0\\
0&\overline{\gamma}\end{array}\right]$ we arrive to the pair $(A,B)$ as stated in the proposition.
\end{proof}

If $A,B\in\bh$ and $A\po B$, then 
$$H:=\left[\begin{array}{cc}
0&A\\
A^*&0\end{array}\right]\ \ \mbox{and}\ \ K:=\left[\begin{array}{cc}
0&B\\
B^*&0\end{array}\right]$$
are selfadjoint and $H\po K$. However, two nonzero selfadjoint operators on $\bc^n$ can not be Pythagoras orthogonal if one of them is positive. This is a consequence of the following  lemma.

\begin{lemma}\label{00}If $A,B\in{\rm B}(\bc^n)$ are selfadjoint, $B\geq0$ and $\det(A+\lambda B)=0$ for all $\lambda\in\bc$, then $\ker A\cap\ker B\ne0$. Hence $A$ and $B$ are simultaneously unitarily similar to matrices that both have the last column and the last row equal to $0$.
\end{lemma}

\begin{proof}Replacing $A$ and $B$ by $S^*AS$ and $S^*BS$, where $S\in\mat{n}{\bc}$ is invertible, does not change the problem, hence we may assume that $B$ is a projection, so that relative to the decomposition $\bc^n=B(\bc^n)\oplus\ker B$ the two operators are
$$A=\left[\begin{array}{cc}
A_1&A_2\\
A_2^*&A_3\end{array}\right]\ \ \mbox{and}\ \ B=\left[\begin{array}{cc}
I&0\\
0&0\end{array}\right].$$
We must prove that there exists a nonzero vector in $$\ker\left[\begin{array}{c}A_2\\
 A_3\end{array}\right],$$
or equivalently, that the rank of this matrix is less than the number of its columns. If we multiply $A$ and $B$ from the left by an invertible block-diagonal matrix of the form $U=I\oplus V$ and from the right by $U^*$, this does not change the problem since 
$$\left[\begin{array}{c}
A_2\\
A_3\end{array}\right]\ \mbox{transforms into}\ \left[\begin{array}{cc}
I&0\\
0&V\end{array}\right]\left[\begin{array}{c}
A_2\\
A_3\end{array}\right]V^*,$$
which does not change the rank. Thus we may assume (using an appropriate choice of $V$) that $A_3$ is diagonal, of the form $A_3=I\oplus0$, so that $A$ and $B$ have now the form
$$A=\left[\begin{array}{ccc}
A_1&C_1&C_2\\
C_1^*&I&0\\
C_2^*&0&0\end{array}\right],\ \ B=\left[\begin{array}{ccc}
I&0&0\\
0&0&0\\
0&0&0\end{array}\right].$$
Multiplying the matrix
$$A+\lambda B=\left[\begin{array}{ccc}
\lambda I+ A_1&C_1&C_2\\
C_1^*&I&0\\
C_2^*&0&0\end{array}\right]$$
from the left by the matrix
$$P=\left[\begin{array}{ccc}
I&-C_1&0\\
0&I&0\\
0&0&I\end{array}\right]$$
and from the right by $P^*$, transforms $A+\lambda B$ into
$$A(\lambda):=\left[\begin{array}{ccc}
\lambda I+A_1-C_1C_1^*&0&C_2\\
0&I&0\\
C_2^*&0&0\end{array}\right].$$
By this operation the sub-matrix consisting of the last two block-columns of $A+\lambda B$ has been simply multiplied from the left by $P$, which can not change the rank. Now we have
\begin{equation}\label{01}0=\det A(\lambda)=\det\left[\begin{array}{cc}
\lambda I+A_1-C_1C_1^*&C_2\\
C_2^*&0\end{array}\right].\end{equation}
For large enough $\lambda$ the matrix $\lambda I+A_1-C_1C_1^*$ is invertible, hence (\ref{01}) implies that
$$0=\det\left[\begin{array}{cc}
I&0\\
-C_2^*(\lambda I+A_1-C_1C_1^*)^{-1}&I\end{array}\right]\det\left[\begin{array}{cc}
\lambda I+A_1-C_1C_1^*&C_2\\
C_2^*&0\end{array}\right]$$$$=\det\left[\begin{array}{cc}
\lambda I+A_1-C_1C_1^*&C_2\\
0&-C_2^*(\lambda I+A_1-C_1C_1^*)^{-1}C_2\end{array}\right]$$$$=-\det(\lambda I+A_1-C_1C_1^*)\det(C_2^*(\lambda I+A_1-C_1C_1^*)^{-1}C_2).$$ 
Thus the square matrix $C_2^*(\lambda I+A_1-C_1C_1^*)^{-1}C_2$ is not invertible, hence not injective. Since $\lambda I+A_1-C_1C_1^*\geq0$ if $\lambda\in\br$ is large enough, we can take the square root and it follows that
$(\lambda I+A_1-C_1C_1^*)^{-1/2}C_2$ is not injective, hence $C_2$ is not injective and so the columns of $C_2$ must be linearly dependent.

We have shown that there exists a unit vector in $\ker A\cap\ker B$. Choosing this vector as the last vector of an orthonormal basis of $\bc^n$, we can represent operators $A$ and $B$ by matrices that have the last row and the last column identically $0$.
\end{proof}

\begin{pr}\label{pr5}If $A,B\in\mat{n}{\bc}$ are nonzero, selfadjoint and $B\geq0$, then $A\not\perp_PB$.
\end{pr}

\begin{proof}Assume the contrary, that $A\po B$. We may suppose that $\|A\|=1=\|B\|$. Then $\det(A+\lambda B)=0$ for all $\lambda\in\bc$ by Proposition \ref {pr2}, hence by Lemma \ref{00} we may suppose  that $A$ and $B$ have the last row and the last column equal to $0$. In other words, $A=A_1\oplus0$ and $B=B_1\oplus0$, where $A_1$ and $B_1$ are self-adjoint matrices of size $(n-1)\times(n-1)$, with $B_1\geq0$ and $\|A_1\|=1=\|B_1\|$. It is easy to verify that $A_1\perp_PB_1$, hence this reduces the problem to one dimension smaller. Continuing in this way we arrive after finite number of steps at $1\times 1$ matrices $\alpha,\beta$, that is $\alpha,\beta\in\br$, such that $\beta=1=|\alpha|$ and $\alpha\perp_P\beta$. This means that $|\alpha+\lambda|^2=|\alpha+\lambda \beta|^2=1+|\lambda|^2$ for all $\lambda\in\bc$, which is clearly a contradiction.
\end{proof}

If $P\in\bh$ is a projection and $U\in\bh$ is a partial isometry such that $U^*U=P$ and $UU^*=P^{\perp}=I-P$, then it is easy to verify that $P\po U$. Proposition \ref{pr4} shows that this is essentially the only example of a Pythagoras orthonormal pair on $\h=\bc^2$. The previous two propositions could lead to an impression that  orthogonal pairs of operators are rare, but in fact, already on $\bc^3$ there are many such pairs.  Below we present two classes of examples that were found with the help of Proposition \ref{pr2} and Lemma \ref{le1}. Additional example, namely, the description of all operators that are Pythagoras orthogonal to a projection of rank $1$, is postponed to Section 6, since it requires a lot of computation.

\begin{ex}\label{ex}(i) Here is an example, where the kernels of $A$ and $B$ have only $0$ in their intersection and the same holds for the intersection of kernels of $A^*$ and $B^*$.
$$A=\left[\begin{array}{ccc}
0&0&1\\
0&\alpha&0\\
0&\beta&0\end{array}\right],\ \ B=\left[\begin{array}{ccc}
1&0&0\\
0&b&0\\
0&0&0
\end{array}\right],$$ where $0<b<1$, $\beta\ne0$  and $|\alpha|^2\leq(1-|\beta|^2)(1-b^2)$ (hence also $|\alpha|^2+|\beta|^2\leq1$).
Let us verify that this pair of operators satisfies the condition for orthogonality in Lemma \ref{le1}. Indeed, $\|A\|=1=\|B\|$ and the eigenvalues of the matrix
$$(|\lambda|^2+1)I-(A+\lambda B)^*(A+\lambda B)=$$$$\left[\begin{array}{ccc}
1&0&-\overline{\lambda}\\
0&(1-b^2)|\lambda|^2-\overline{\alpha}b\lambda-\alpha b\overline{\lambda}+1-|\alpha|^2-|\beta|^2&0\\
-\lambda&0&|\lambda|^2\end{array}\right]$$
are $0$, $|\lambda|^2+1$ and
$(1-b^2)|\lambda|^2-\overline{\alpha}b\lambda-\alpha b\overline{\lambda}+1-|\alpha|^2-|\beta|^2$. Thus we have only to show that the last eigenvalue is non-negative for all $\lambda\in\bc$. This is equivalent to the fact that the matrix
$$\left[\begin{array}{cc}
1-b^2&-\overline{\alpha}b\\
-\alpha b&1-|\alpha|^2-|\beta|^2\end{array}\right]$$
is positive, since its diagonal terms and the determinant are non-negative.

(ii) An example of a Pythagorean orthogonal pair in ${\rm M}_{3,2}(\bc)$ is
$$A=\left[\begin{array}{cc}
0&0\\
u\sin\phi\sin\psi&w\sin\phi\cos\psi\\
v\cos\psi&-\overline{u}vw\sin\psi\end{array}\right],\ \ \, B=\left[\begin{array}{cc}
1&0\\
0&\cos\phi\\
0&0\end{array}\right],$$
where $u,v,w\in\bc$ have absolute value $1$ and $\phi,\psi\in\br.$ Namely, it can be shown by a routine (although somewhat lengthy) computation, which we will omit, that this pair satisfies the criterion of Lemma \ref{le1}. Here both $A$ and $B$ can have rank $2$. 
\end{ex}

\section{Pythagoras orthogonal pairs of commuting normal operators}

If $A$ and $B$ are commuting normal operators then there is a polar decomposition $B=U|B|$ of $B$, where $U$ is a unitary in the abelian W$^*$-algebra generated by $A$, $B$ and the identity. Then the operators $U^*A$ and $|B|$ commute and are normal, hence in studying the Pythagoras orthogonality for such operators there is no loss of generality in assuming that one of the operators is positive.

Recall  \cite[p. 22]{Mu} that the (joint) spectrum $\sigma(A,B)$ of two commuting normal operators is defined as $$\sigma(A,B)=\{(\omega(A),\omega(B)):\ \omega\in\Delta\},$$
where $\Delta$ denotes the set of all multiplicative functionals on the C$^*$-algebra $\ca$ generated by $A$ and $B$.

\begin{pr}\label{nor}Let $A,B\in\bh$ be commuting normal operators with $\|A\|=1=\|B\|$ and $B\geq0$. Then $A\po B$ if and only if the  spectrum $\sigma(A,B)$ is contained in the unit half-ball 
$\{(\zeta,t)\in\bc\times\br:\, |\zeta|^2+t^2\leq1,\, t\geq0\}$ and  contains the hemisphere $\{(\zeta,t)\in\bc\times\br:\, |\zeta|^2+t^2=1,\, t\geq0\}$. 
\end{pr}

\begin{proof}Since the norm of a normal operator is equal to its spectral radius  the condition $A\po B$ is equivalent to $$\max_{\omega\in\Delta}|\omega(A+\lambda B)|^2=1+|\lambda|^2\ \ \forall\lambda\in\bc.$$
This condition means (using also $B\geq0$) that 
\begin{equation}\label{90}(1-\omega(B)^2)|\lambda|^2-2\rea(\lambda\omega(A^*B))+1-|\omega(A)|^2\geq0\ \ \forall\lambda\in\bc,\ \forall\omega\in\Delta\end{equation}
and that for each $\lambda\in\bc$ there exist $\omega_{\lambda}\in\Delta$ such that equality holds in (\ref{90}) when $\omega=\omega_{\lambda}$.  If $\omega(B)\ne1$, we can write (\ref{90}) as
\begin{equation}\label{91}|\sqrt{1-\omega(B)^2}\lambda-\frac{\omega(B^*A)}{\sqrt{1-\omega(B)^2}}|^2+1-|\omega(A)|^2-\frac{|\omega(B^*A)|^2}{1-\omega(B)^2}\geq0\ \ \forall{\lambda\in\bc},\end{equation} which means that $1-|\omega(A)|^2-\frac{|\omega(B^*A)|^2}{1-\omega(B)^2}\geq0$, that is
 \begin{equation}\label{92}|\omega(A)|^2+\omega(B)^2\leq1.\end{equation}
This holds even if $|\omega(B)|=1$, for in this case (\ref{90}) implies (by considering $|\lambda|\to\infty$) that $\omega(A)=0$. Further, since for $\omega=\omega_{\lambda}$ equality holds in (\ref{90}), $\omega_{\lambda}(B)\ne1$, for otherwise (\ref{92}) would imply that $\omega_{\lambda}(A)=0$ and then the equality could not hold in (\ref{90}). Hence equality must hold also in (\ref{91}). Thus for $\omega=\omega_{\lambda}$ equality holds in (\ref{92}) and in (\ref{91}). Using $|\omega_{\lambda}(A)|^2+\omega_{\lambda}(B)^2=1$ we can  simplify the equality case of (\ref{91}) when $\omega=\omega_{\lambda}$ to $\omega_{\lambda}(B)=\overline{\lambda}\omega_{\lambda}(A)$, hence we have now
\begin{equation}\label{93}\forall\lambda\in\bc\ \exists\omega_{\lambda}\in\Delta\ \mbox{such that}\ |\omega_{\lambda}(A)|^2+|\omega_{\lambda}(B)|^2=1\ \mbox{and}\ \omega_{\lambda}(B)=\overline{\lambda}\omega_{\lambda}(A).\end{equation}
The inequality (\ref{92}), together with $B\geq0$, means that $\sigma(A,B)$ is contained in the half-ball as stated in the theorem, while (\ref{93}) means that $\sigma(A,B)$ contains a point $(\zeta,t)$ in the intersection of the hemisphere with the ray $t=\overline{\lambda}\zeta$. Each point in the hemisphere $\{(\zeta,t)\in\bc\times\br:\, |\zeta|^2+t^2=1,\, t\geq0\}$ is on such a ray, except the north pole $(0,1)$, but $\sigma(A,B)$ is closed, hence it must contain the whole hemisphere.
\end{proof}

\begin{co}Two non-zero commuting normal operators $A,B$ on a finite dimensional Hilbert space can not be Pythagoras orthogonal.
\end{co}

\begin{proof}The fact that C$^*(A,B)$ is finite dimensional (and abelian) implies that C$^*(A,B)$ has only finitely many multiplicative functionals, consequently $\sigma(A,B)$ is a finite set. Hence $\sigma(A,B)$ can not contain the hemisphere $\{(\zeta,t)\in\bc\times\br:|\zeta|^2+t^2=1,\ t\geq0\}$, so by Proposition \ref{nor} $A$ and $B$ can not be Pythagoras orthogonal.
\end{proof}

The joint (algebraic) numerical range of an $n$-tuple $(A_1,\ldots,A_n)$ of elements of a C$^*$-algebra $\mathcal{A}$  is defined as
$$V(A_1,\ldots,A_n)=\{(\omega(A_1),\ldots,\omega(A_n)):\, \omega\in S(\ca)\},$$
where $S(\ca)$ is the set of all states on $\mathcal{A}$ (= positive functionals of norm $1$). Pythagoras orthogonality of general operators can be characterized as follows:

\begin{pr}\label{th1} Let $A,B\in\bh$ and $\|A\|=1=\|B\|$. Then $A\po B$ if and only if the set $V:=V(I-A^*A,I-B^*B,B^*A)$ is contained in the ``cone''
$$\mathcal{C}=\{(x,y,z)\in\br^2\times\bc:\, 0\leq x\leq1,\, 0\leq y\leq1,\, |z|\leq\sqrt{xy}\}$$
and for all $w\in\bc$ with $|w|=1$
and $s,t\in\br_+=(0,\infty)$ the set
$V$ intersects the closed ray from $0$ in the direction of vector $(s,t,w\sqrt{st})$.
\end{pr}

\begin{proof}[The idea of the proof] Since $\|T\|^2=\|T^*T\|=\max\{\omega(T^*T):\, \omega\in S(\bh)\}$ for each $T\in\bh$, the condition 
$$\|A+\lambda B\|^2=\|(A+\lambda B)^*(A+\lambda B)\|=1+|\lambda|^2$$
is equivalent to the simultaneously validity of the following two conditions:
$$\omega((A+\lambda B)^*(A+\lambda B))\leq1+|\lambda|^2\ \ \forall\lambda\in\bc,\ \forall\omega\in S(\bh)\ \ \ \mbox{and}$$
$$\forall\lambda\in\bc\ \exists\omega_{\lambda}\in S(\bh)\ \mbox{such that}\ \omega_{\lambda}((A+\lambda B)^*(A+\lambda B))=1+|\lambda|^2.$$

Using these two conditions the proof can be accomplished by an elementary computation, analogous to  the one in the proof of Proposition \ref{nor}.
Since the proposition will not be used later in the paper we will omit this details. 
\end{proof}

\section{Column orthogonal operators}

The tensor product $\matn{\bc}\otimes\bh$ used below is the usual tensor product with the (unique C$^*$-tensor) norm that comes from the natural isomorphisms $\matn{\bc}\otimes\bh\cong\matn{\bh}\cong{\rm B}(\h^n)$. 
\begin{de}\label{dec}A finite set of  operators $B_j\in\bh$ is called {\em column orthogonal} if at least one of the operators is $0$ or the operators $C_j:=\|B_j\|^{-1}B_j$ satisfy 
\begin{equation}\label{c1}\|\sum\alpha_j\otimes C_j\|^2=\|\sum\alpha_j^*\alpha_j\|\end{equation}
for all $\alpha_j\in\matn{\bc}$ and all $n\in\bn$.
A general set of operators in $\bh$ is column orthogonal if all of its finite subsets are column orthogonal. 
\end{de}

For a finite set $(C_j)$ the condition (\ref{c1}) means that there is a completely isometric isomorphism $\varphi$ from the linear span of $(C_j)$ into a column Hilbert space $\h$ with an orthonormal set $(\epsilon_j)$ such that $\varphi(C_j)=\epsilon_j$ for all $j$.  (For a formal definition of a column Hilbert space see e. g. \cite{ER}, \cite{Pi}, \cite{Math}.) Thus, in particular, the norm of the row $\left[\begin{array}{lll}C_0&C_1&\ldots\end{array}\right]=\sum_j E_{1,j}\otimes C_j$ (where $E_{i,j}:=\epsilon_i\otimes\epsilon_j^*$ are matrix units) is equal to $\|\left[\begin{array}{lll}\epsilon_0&\epsilon_1&\ldots\end{array}\right]\|=1$, hence 
$$\|\sum C_jC_j^*\|=1.$$
Similarly we could define {\em row orthonormal} set by declaring that for any finite subset $C_0,\ldots,C_m$ there is a complete isometry from ${\rm span}\{C_0,\ldots,C_m\}$ into some row Hilbert space $\h^*$, mapping the $C_j$'s onto orthonormal vectors in $\h^*$. In this case the identity (\ref{c1}) is replaced by $\|\sum\alpha_j\otimes C_j\|^2=\|\sum\alpha_j\alpha_j^*\|$ ($\alpha_j\in\matn{\bc}$) and such operators necessary satisfy $\|\sum C_j^*C_j\|=1$.
The theorem below characterizes column orthonormal sets and the characterization of row orthonormal sets can be obtained then by taking adjoints, but first we need a simple lemma.

\begin{lemma}\label{lec}Let $C_j\in\bk$ satisfy
 $\sum_{j=0}^mC_jC_j^*\leq I$
and let $\omega$ be a state on $\bk$ such that $\omega(C_j^*C_j)=1$ for all $j$. Then $\omega(C_k^*C_j)=0$ if $k\ne j$.
\end{lemma}

\begin{proof}Let $\pi$ be the cyclic representation of $\bk$ on a Hilbert space $\h$ that corresponds to $\omega$ by the GNS construction, $\xi\in\h$ the corresponding cyclic unit vector and $A_j=\pi(C_j)$. Then the hypothesis says that $\sum_{j=0}^mA_jA_j^*=\pi(\sum_{j=0}^mC_jC_j^*)\leq I$ and $\|A_j\xi\|^2=\langle\pi(C_j)^*\pi(C_j)\xi,\xi\rangle=\omega(C_j^*C_j)=1$ for all $j$. Since $\omega(C_k^*C_j)=\langle A_j\xi,A_k\xi\rangle$, we have to prove that $\langle A_j\xi,A_k\xi\rangle=0$ if $j\ne k$ . Let $E:=\xi\otimes\xi^*$ be the projection onto $\bc\xi$ and  $P_j=A_jEA_j^*=A_j\xi\otimes(A_j\xi)^*$. Since $\|A_j\xi\|=1$, each $P_j$ is a projection onto $\bc A_j\xi$. We have $\sum_{j=0}^mP_j=\sum_{j=0}^mA_jEA_j^*\leq\sum_{j=0}^mA_jA_J^*\leq I$, which implies that the ranges of $P_j$ are mutually orthogonal. (Indeed, $\sum_{j=0}^mP_kP_jP_k\leq P_k$ implies that $\sum_{j\ne k}P_kP_jP_k\leq0$. Since $P_kP_jP_k\geq0$, this means that $P_kP_j(P_kP_j)^*=P_kP_jP_k=0$, hence $P_kP_j=0$.) Thus $A_j\xi\perp A_k\xi$ if $k\ne j$. 
\end{proof}

\begin{theorem}\label{thc}A set of norm $1$ operators $C_j\in {\rm B}(\k)$  acting on a  Hilbert space $\k$ is column orthogonal if and only if $\sum C_jC_j^*\leq I$ and there exists a state $\omega$ on ${\rm B}(\k)$ (or on the C$^*$-algebra generated by all $C_j$) such that $\omega(C_j^*C_j)=1$ for all $j$. (Note that the last condition just means that in the cyclic representation $\pi$ arising from $\omega$ the operators $\pi(C_j)$ all attain their norms $1$ at the same vector $\eta$, namely at a cyclic vector for $\pi$, so that $\omega(T)=\langle\pi(T)\eta,\eta\rangle$ for $T\in\bk$.)
\end{theorem}

\begin{proof}Suppose that $\sum_jC_jC_j^*\leq I$ and that there exists a state $\omega$ satisfying $\omega(C_j^*C_j)=1$ for all $j$. For any matrices $\alpha_j,\beta_j\in\matn{\bc}$ and finite subset $F$ of indexes we then have 
$$\|\sum_{j\in F}\alpha_j\otimes C_j\|^2=\|\sum_{j\in F} C_j\otimes\alpha_j\|^2=\|\sum_{j\in F}(C_j\otimes I)(I\otimes\alpha_j)\|^2$$$$=\|\left[\begin{array}{ccc}
C_1\otimes I&C_2\otimes I&\ldots\end{array}\right]\left[\begin{array}{c}I\otimes\alpha_1\\
I\otimes\alpha_2\\
\vdots\end{array}\right]\|^2$$
$$\leq\|\left[\begin{array}{ccc}
C_1\otimes I&C_2\otimes I&\ldots\end{array}\right]\|^2\|\left[\begin{array}{c}I\otimes\alpha_1\\
I\otimes\alpha_2\\
\vdots\end{array}\right]\|^2$$ 
$$\leq\|\sum_{j\in F} C_jC_j^*\|\|\sum_{j\in F}\alpha_j^*\alpha_j\|\leq\|\sum_{j\in F}\alpha_j^*\alpha_j\|.$$
Further, by Lemma \ref{lec}  $\omega(C_k^*C_j)=\delta_{k,j}$, hence 
$$\|\sum_{j\in F}\alpha_j\otimes C_j\|^2=\|(\sum_{j\in F}\alpha_j\otimes C_j)^*(\sum_{j\in F}\alpha_j\otimes C_j)\|$$$$\geq\|({\rm id}\otimes\omega)[(\sum_{j\in F}\alpha_j\otimes C_j)^*(\sum_{j\in F}\alpha_j\otimes C_j)]\|=\|\sum_{j,k\in F}\alpha_k^*\alpha_j\otimes\omega(C_k^*C_j)\|=\|\sum_{j\in F}\alpha_j^*\alpha_j\|.$$
Thus (\ref{c1}) holds and the $C_j$ are column orthonormal.

Suppose now conversely, that the operators $C_j$ are column orthonormal. Assume first that the set is finite, say $C_0,\ldots,C_m$. We have already observed after the Definition \ref{dec} that $\|\sum_jC_jC_j^*\|=1$, thus we have only to prove the existence of an appropriate state $\omega$.   Let $B_j$ ($j=0,\ldots,m$) be rank one operators on $\ell^2$ given in the usual orthonormal basis of $\ell^2$ by matrices
$$B_j=\left[\begin{array}{ccccc}
0&0&\ldots&0&0\\
\vdots&\vdots&\vdots&\vdots&\vdots\\
1&0&\ldots&0&0\\
0&0&\ldots&0&0\\
\vdots&\vdots&\vdots&\vdots&\vdots
\end{array}\right],$$
where $1$ is in the $j$-th row.  Since the operator sets $(B_j)_{j=0}^m$ and $(C_j)_{j=0}^m$ are both column orthonormal, there exists a completely isometric isomorphism $\varphi: {\rm span}(C_j)\to{\rm span}(B_j)$ such that $\varphi(C_j)=B_j$. As a complete contraction, $\varphi$ is necessarily of the form $\varphi(T)=X\pi(T)Y$ for suitable contractions $X$ and $Y$ and a representation $\pi$ of the C$^*$-algebra $\ca$ generated by $(C_j)_{j=0}^m$ on a Hilbert space $\h$ \cite[pp. 99 and 102]{Pa}.  Let $A_j:=\pi(C_j)$, so that 
\begin{equation}\label{c4}B_j=XA_jY\ \ (j=0,\ldots,m).\end{equation}
Denote $P:=B_0$ (a projection). Since  $B_j=B_jP$, we may replace in (\ref{c4}) $Y$ by $YP$, so the operators $X:\h\to\ell^2$ and $Y:\ell^2\to\h$ are of the form
$$X=\left[\begin{array}{c}
\xi_0^*\\
\xi_1^*\\
\xi_2^*\\
\vdots
\end{array}\right],\ \ Y=\left[\begin{array}{cccc}
\eta&0&0&\ldots\end{array}\right],$$  
where $\eta:\bc\to\h$ is essentially a vector in $\h$ and the $\xi_j^*:\h\to\bc$ are contractive linear functionals on $\h$, hence given by $\xi_j^*(\zeta)=\langle\zeta,\xi_j\rangle$ ($\zeta\in\h$) for some vectors $\xi_j\in\h$ with $\|\xi_j\|\leq1$. Comparing the entries of matrices in  (\ref{c4}) we have 
now 
\begin{equation}\label{c7}\langle A_j\eta,\xi_j\rangle=1\ \ \mbox{and}\ \ \langle A_j\eta,\xi_k\rangle=0,\ \mbox{if}\ k\ne j.\end{equation}
Since $\|A_j\eta\|\leq1$ and $\|\xi_j\|\leq1$, the first equality in (\ref{c7}) implies that $\|\xi_j-A_j\eta\|^2\leq2-2\rea(\langle A_j\eta,\xi_j\rangle)=0$, hence $\xi_j=A_j\eta$. Thus it follows from (\ref{c7}) that $\langle A_j\eta,A_k\eta\rangle=\delta_{k,j}$. Hence the state $\omega$ on $\bk$, defined by $\omega(T)=\langle\pi(T)\eta,\eta\rangle$ satisfies $\omega(C_k^*C_j)=\delta_{k,j}$.  This completes the the proof in the case when the set $(C_j)$ is finite. If the set of operators $(C_j)$ is infinite, we can apply the argument just given to each of its finite subsets $F$ to obtain a state $\omega_F$ satisfying $\omega_F(C_k^*C_j)=\delta_{j,k}$ for all $j,k\in F$, and then we take a weak* limit point $\omega$ of the net of states $\omega_F$. 
\end{proof}

Recall that each state on the C$^*$-algebra of compact operators ${\rm K}(\k)$ is of the form $T\mapsto\sum_{j=1}^{\infty}\langle T\xi_j,\xi_j\rangle$, where $\xi_j\in\k$ and $\sum_j\|\xi_j\|^2=1$. (If $\dim\k<\infty$ the sum can be taken to have only finitely many terms, so each state  is a convex combination of vector states.) Further, each state on  $\bk$ can be approximated by vector states, hence Theorem \ref{thc} implies the following corollary.
\begin{co}\label{coc}A set of norm $1$ operators $C_j$ ($j=0,\ldots,m$) on $\k=\bc^r$ ($r\in\bn$) is column orthogonal if and only if $\|\sum_{j=0}^mC_jC_j^*\|=1$ and all the operators $C_j$ achieve  their norms at the same unit vector in $\k$. The same conclusion holds for compact operators on an infinite dimensional Hilbert space $\k$. For general operators a similar conclusion holds, but the norm attaining condition must be replaced by:
for each $\varepsilon>0$ there exists a unit vector $\xi\in\k$ such that $\|C_j\xi\|>1-\varepsilon$ for all $j$. 
\end{co}

For norm $1$ operators $C_j$ and unit vector $\xi$ the norming condition $\|C_j\xi\|\approx1$ in Corollary \ref{coc} is equivalent to the requirement that the norm of the column $(C_0,\ldots,C_m)^T$ is $\sqrt{m}$. Since in the proof of Theorem \ref{thc} we have used only the row and the column structure, it follow in particular that the operator space structure of the column Hilbert space is determined already by the norms on spaces of rows and columns (columns with orthonormall entries are sufficient), which, however, has been proved already by Mathes \cite{Math}.

\section{A metric characterization of operators with orthogonal ranges}

In this section we will study a more restrictive form of Pythagoras orthogonality, which turns out to be also a special case of column orthogonality, in which scalars are replaced by elements of a C$^*$-algebra $\ca$. For this, we will need a metric characterization of pairs $A,B\in\ca$ satisfying $A^*B=0$. First a lemma is needed, which (as pointed to me by an anonymous referee) follows from \cite[Lemma 2.3]{FMX}, but we will present a short direct proof.

\begin{lemma}\label{lero}Let $A,B\in\ca$, where $\ca$ is a C$^*$-algebra. Then $B^*B\leq A^*A$ if and only if $\|BX\|\leq\|AX\|$ for all (positive) $X\in\ca$.
\end{lemma}

\begin{proof}If $B^*B\leq A^*A$, then $X^*B^*BX\leq X^*A^*AX$ for all $X\in\ca$, which implies that $\|BX\|^2=\|X^*B^*BX\|\leq\|X^*A^*AX\|=\|AX\|^2$. To prove the converse, suppose that $B^*B\not\leq A^*A$.
Then there exists $t\in(0,1)$ such that the positive part $X$ of the operator $tB^*B-A^*A$ is not zero, that is \begin{equation}\label{ro5}X:=(tB^*B-A^*A)_+\ne0.\end{equation}
(Otherwise $tB^*B-A^*A\leq0$ for all $t\in(0,1)$ and letting $t\to1$ it would follow that $B^*B\leq A^*A$). Then $X\geq0$ and 
\begin{equation}\label{511}0\ne X(tB^*B-A^*A)X\geq0.\end{equation}
Hence $XA^*AX\leq tXB^*BX$ and therefore $$\|AX\|^2=\|XA^*AX\|\leq t\|XB^*BX\|=t\|BX\|^2<\|BX\|^2,$$
since $BX\ne0$. (Namely, if $BX=0$, then (\ref{511}) would be a contradiction.)
\end{proof}

\begin{theorem}\label{prro}For elements $A,B$ in any C$^*$-algebra $\ca$ the equality $A^*B=0$ holds if and only if
\begin{equation}\label{ro}\|AX\|\leq\|AX+BY\|\ \ \forall X,Y\in\ca.\end{equation}
\end{theorem}

\begin{proof}We may replace $X$ and $Y$ in (\ref{ro}) by $XC$ and $YC$ for any $C\in\ca$ and then apply Lemma \ref{lero} to $AX$ and $AX+BY$ (instead of $A$ and $B$). In this way we see that (\ref{ro}) holds if and only if 
$$C^*X^*A^*AXC\leq C^*(AX+BY)^*(AX+BY)C\ \ \forall X,Y,C\in\ca,$$
which can be rewritten as
$$2\rea(C^*X^*A^*BYC)+C^*Y^*B^*BYC\geq0.$$
Replacing $Y$ by $twY$, where $w\in\bc$ with $|w|=1$ and $t\in(0,\infty)$, we obtain equivalent condition
$$2\rea(wC^*X^*A^*BYC)+tC^*Y^*B^*BYC\geq0.$$
Considering $t\to0$, we see that
$$\rea(w(C^*X^*A^*BYC))\geq0\ \ \forall w\in\bc\ \mbox{with}\ |w|=1.$$
This means that $C^*X^*A^*BYC=0$ (to see this, apply states of $\ca$ to $C^*X^*A^*BYC$) and, since $C, X$ and $Y$ are arbitrary,  $A^*B=0$. The verification of converse is easy.
\end{proof}

\begin{co}\label{coro}Let $P_j$ ($j=0,\ldots,m$) be fixed positive elements in a C$^*$-algebra $\ca$. For elements $A_j$ in $\ca$ the equality
\begin{equation}\label{ro2}\|\sum_jA_jX_j\|^2=\|\sum_jX_j^*P_jX_j\|\end{equation}
holds for all $X_j\in\ca$ if and only if $A_j^*A_j=P_j$ and $A_k^*A_j=0$ if $k\ne j$. In particular, if all $P_j$ are equal to a projection $P$, then the $A_j$ are partial isometries with orthogonal ranges and the same initial projection $P$. 
\end{co}

\begin{proof} If $A_j$ satisfy (\ref{ro2}) for all $X_j\in\ca$, then (taking $X_j=0$ for $j\ne0$) we get $$\|X_0^*A_0^*A_0X_0\|=\|A_0X_0\|^2=\|X_0^*P_0X_0\|$$ for all $X_0\in\ca$, hence  by Lemma \ref{lero} (applied to $A=A_0$ and $B=\sqrt{P_0}$) $A_0^*A_0=P_0$. Similarly $A_j^*A_j=P_j$ for all $j$. Further, from (\ref{ro2}) we now have
$$\|A_jX_j+A_kX_k\|^2=\|X_j^*P_jX_j+X_k^*P_kX_k\|$$$$\geq\|X_k^*P_kX_k\|=\|X_k^*A_k^*A_kX_k\|=\|A_kX_k\|^2$$
for all $X_k,Y_k\in\ca$ and all $k\ne j$, hence by Theorem \ref{prro} $A_k^*A_j=0$. This proves the corollary in one direction, while the proof in the reverse direction is straightforward:
if $A_j^*A_j=P_j$ and $A_k^*A_j=0$ for $k\ne j$, then $$\|\sum_jA_jX_j\|^2=\|(\sum_j A_jX_j)^*(\sum_kA_kX_k)\|=\|\sum_jX_j^*P_jX_j\|.$$
\end{proof}

\begin{pr}\label{prro1}Suppose that $A,B\in\ca$ satisfy $A^*B=0$. Then $A\po B$ if and only if there exists a state $\omega$ on $\ca$ such that $\omega(A^*A)=\|A\|^2$ and $\omega(B^*B)=\|B\|^2$. In this case $A$ and $B$ are column orthogonal.
\end{pr}

\begin{proof}Since $A^*B=0=B^*A$, 
\begin{equation}\label{ro4}\|\alpha\otimes A+\beta\otimes B\|^2=\|\alpha^*\alpha\otimes A^*A+\beta^*\beta\otimes B^*B\|\ \ \forall\alpha,\beta\in\matn{\bc}.\end{equation}
If $A\po B$, then $\|\alpha A+\beta B\|^2=|\alpha|^2\|A\|^2+|\beta|^2\|B\|^2$ for all $\alpha,\beta\in\bc=\mat{1}{\bc}$, hence it follows from (\ref{ro4})  that
$\||\alpha|^2A^*A+|\beta|^2B^*B\|=|\alpha|^2\|A\|^2+|\beta|^2\|B\|^2$ for all $\alpha,\beta\in\bc$. Denoting $s=|\alpha|^2, t=|\beta|^2$, $P=A^*A$ and $Q=B^*B$, this is equivalent to 
\begin{equation}\label{ro50}\|s P+tQ\|=s\|P\|+t\|Q\|\ \ \forall s,t\in\br_+.\end{equation}
Let $s>0, t>0$ be fixed. Since for positive operators norm is equal to the numerical radius, we can choose a state $\omega$ on $\ca$ such that $\omega(sP+tQ)=\|s P+tQ\|$, and then we have
$$\|sP+tQ\|=s\omega(P)+t\omega(Q)\leq s\|P\|+t\|Q\|.$$ 
Here equality holds by (\ref{ro50}), hence it follows that $\omega(P)=\|P\|$ and $\omega(Q)=\|Q\|$, that is, $\omega(A^*A)=\|A\|^2$ and $\omega(B^*B)=\|B\|^2$. 

Conversely, if there exists a state $\omega$ satisfying $\omega(A^*A)=\|A\|^2$ and $\omega(B^*B)=\|B\|^2$ and $A\ne0\ne B$, then by Theorem \ref{thc} $A_0:=\|A\|^{-1}A$ and $B_0:=\|B\|^{-1}B$ are column orthonormal, since $(A_0A_0^*+B_0B_0^*)^2=A_0(A_0^*A_0)A_0^*+B_0(B_0^*B_0)B_0^*
\leq A_0A_0^*+B_0B_0^*$
implies that $A_0A_0^*+B_0B_0^*\leq I$. 
\end{proof}

\section{Operators orthogonal to a projection of rank one}

In this section we will determine (up to equivalence) all operators $A$ with $\|A\|=1$ that are Pythagoras orthogonal to a projection $B$ of rank one. We will see that Pythagoras orthogonality can depend on the action of $A$ on the entire Hilbert space. 
If $A\po B$, then by Lemma  \ref{le3}, relative to the decomposition $\h=B\h\oplus(1-B)\h$, $A\in\bh$ is represented by a matrix of the form 
$$A=\left[\begin{array}{cc}
0&a^*\\
b&C\end{array}\right],\ \mbox{hence}\ A+\lambda B=\left[\begin{array}{cc}
\lambda&a^*\\
b&C\end{array}\right],$$
where $a,b\in(I-B)\h={\rm B}(\bc,(I-B)\h)$. (Here $a^*:(I-B)\h\to\bc$ acts as $a^*(\zeta)=\langle\zeta,a\rangle$).
Thus 
\begin{equation}\label{11}(|\lambda|^2+1)I-(A+\lambda B)^*(A+\lambda B)=\left[\begin{array}{cc}
1-\|b\|^2&-\overline{\lambda}a^*-b^*C\\
-\lambda a-C^*b&(|\lambda|^2+1)I-aa^*-C^*C\end{array}\right].\end{equation}
By Lemma \ref{le1} $A\po B$ if and only if this matrix is positive and singular. If $\|b\|=1$, then the off-diagonal terms of the matrix (\ref{11}) must be $0$ for all $\lambda$ by positivity of the matrix, hence $a=0$ and $C^*b=0$. If $\|b\|<1$, we may multiply the matrix (\ref{11}) from the left by the matrix
$$S:=\left[\begin{array}{cc}
1&0\\
\gamma(\lambda a+C^*b)&I\end{array}\right],\ \ \mbox{where}\ \gamma=\frac{1}{1-\|b\|^2},$$
and form the right by $S^*$ to obtain
$$\left[\begin{array}{cc}
\gamma^{-1}&0\\
0&F(\lambda)\end{array}\right],\ \ \mbox{where}\ F(\lambda):=(|\lambda|^2+1)I-aa^*-C^*C-\gamma(\lambda a+C^*b)(\overline{\lambda}a^*+b^*C).$$
This does not change the non-invertibility and positivity, hence $A\po B$ if and only if $F(\lambda)$ is positive and not invertible.  Since $\|A\|=1$, the last column of $A$ is a contraction, that is, 
\begin{equation}\label{22}aa^*+C^*C\leq I,\end{equation} so that  the operator 
$$E(\lambda):=(|\lambda|^2+1)I-aa^*-C^*C$$
is invertible and positive if $\lambda\ne0$. We may write
\begin{equation}\label{000}F(\lambda)=E(\lambda)^{1/2}[I-D(\lambda)]E(\lambda)^{1/2},\end{equation}
where
$$D(\lambda)=\gamma E(\lambda)^{-1/2}(\lambda a+C^*b)(\overline{\lambda}a^*+b^*C)E(\lambda)^{-1/2}.$$
Observe that $D(\lambda)$ is a rank $1$ operator of the form $cc^*$ (where $c\in\h$) and each such operator has only two eigenvalues, namely $0$ and $c^*c=\|c\|^2$ (since $(cc^*)c=\|c\|^2c$). From (\ref{000}) we see that $F(\lambda)$ is not invertible and positive if and only if   $1$ is an eigenvalue of $D(\lambda)$ (this implies that $D(\lambda)\leq1$ since the only other eigenvalue of $D(\lambda)$ is $0$). Hence it follows that $A\po B$ if and only if 
$$\gamma\|E(\lambda)^{-1/2}(\lambda a+C^*b)\|^2=1.$$
This can be written as
\begin{equation}\label{12}(\lambda a+C^*b)^*E(\lambda)^{-1}(\lambda a+C^*b)=\frac{1}{\gamma}=1-\|b\|^2.\end{equation}
Now observe (by considering $F(\lambda)/|\lambda|^2$ as $|\lambda|\to\infty$) that the non-invertibility of $F(\lambda)$  implies that $I-\gamma aa^*$ is not invertible, hence $\gamma\|a\|^2=1$, that is $$\|a\|^2+\|b\|^2=1.$$ We can now write (\ref{12}) as
$$(\overline{\lambda}a^*+b^*C)E(\lambda)^{-1}(\lambda a+C^*b)=\|a\|^2\ \ (|\lambda|\ne0).$$
or, replacing $\lambda$ by $\frac{1}{\lambda}$,
\begin{equation}\label{13} (a^*+\overline{\lambda}b^*C)(I+|\lambda|^2T)^{-1}(a+\lambda C^*b)=\|a\|^2,\ \mbox{where}\ T=I-aa^*-C^*C.\end{equation}
For $\lambda\in\bc$ satisfying $|\lambda|\|T\|<1$ we can expand $(I+|\lambda|^2T)^{-1}$ and rewrite (\ref{13}) as
\begin{equation}\label{14}(a^*+\overline{\lambda}b^*C)\left(I-|\lambda|^2T+|\lambda|^4T^2-|\lambda|^6T^3+\ldots\right)(a+\lambda C^*b)=\|a\|^2.\end{equation}
Looking at coefficients of various powers of $\lambda$ and $\overline{\lambda}$ we see that
\begin{equation}\label{15}a^*T^nC^*b=0\ \ (n=0,1,2,\ldots) \ \ \mbox{and}\end{equation}
\begin{equation}\label{16}b^*CT^nC^*b=a^*T^{n+1}a\ \ (n=0,1,2,\ldots).\end{equation}
If $\|b\|=1$, then we may apply the above arguments to $A^*$ instead of $A$, which shows that in this case the identities (\ref{14}) and (\ref{15}) holds with the roles of $a$ and $b$ interchanged. (Recall also that $a=0$ if $\|b\|=1.$)

Now assume that $\|b\|\ne1$ and, to simplify further arguments, observe that we may initially replace $A$ and $B$ by equivalent operators of the form $SAT$ and $SBT=B$, where $S=1\oplus U$ and $T=1\oplus V$ are unitary, hence we may assume that $C$ is positive (and diagonal if $\dim\h<\infty$). Then the identity (\ref{15}) can also be written as $\langle T^na,Cb\rangle=0$, which means that the cyclic subspaces $[C^*(T)a]$ and $[C^*(T)Cb]$ are orthogonal. (Here $C^*(T)$ is the C$^*$-algebra generated by $T$, which is just the closure of  polynomials in $T$ since $T^*=T$.) Further, (\ref{16})  can be written as
$$\langle T^kCb,T^lCb\rangle=\langle T^kT^{1/2}a,T^lT^{1/2}a\rangle\ \ (k,l\in\bn),$$
which implies that there is a unique surjective isometry $V:[C^*(T)T^{1/2}a]\to[C^*(T)Cb]$
satisfying $VT^kT^{1/2}a=T^kCb$. In other words, $VT^{1/2}a=Cb$ and $VT=TV$. Decomposing $\h$ as $\h=[C^*(T)a]\oplus[C^*(T)Cb]\oplus\k$ (where by definition $\k$ is the orthogonal complement of the first two summands), $T$ is represented by a block diagonal matrix of the form $T_1\oplus T_2\oplus T_3$. If we define the unitary operator $U$ on $\h$ by 
$$U=\left[\begin{array}{ccc}
0&V^*&0\\
V&0&0\\
0&0&I\end{array}\right],$$
then $U=U^*$ commutes with $T$ and $UT^{1/2}a=Cb$. This is  true even if $\|b\|=1$, since we have already established (from the positivity of the matrix (\ref{11})) that in this case  $a=0$ and $Cb=0$. This proves in one direction the following proposition.

\begin{pr}\label{r1}Up to isometries  of $\bh$ all operators $A\in\bh$ with $\|A\|=1$ that satisfy $A\po B$, where $B\in\bh$ is a projection of rank one, are, relative to the decomposition $\h=B\h\oplus\ker B$, of the form
\begin{equation}\label{501}
A=\left[\begin{array}{cc}
0&a^*\\
b&C\end{array}\right],\end{equation}
where $\|a\|^2+\|b\|^2=1$, $C\geq0$ and $Cb=UT^{1/2}a$ for a self-adjoint unitary $U$ satisfying $UT=TU$ and $U[C^*(T)T^{1/2}a]\perp[C^*(T)a]$, where $T=I-aa^*-C^2$. (Thus in particular (\ref{15}) and (\ref{16}) together are equivalent to (\ref{13}).)
\end{pr}

\begin{proof}By the above arguments we only need to verify that the identity (\ref{13}) holds if $Cb=UT^{1/2}a$ and $C^*=C$, where $U$ and $T$ are as in the proposition. Using the definition of $T$, the identity (\ref{13}), which we need to verify, can be rewritten as
\begin{equation}\label{51}\langle(I+|\lambda|^2T)^{-1}(I+\lambda UT^{1/2})a,(I+\lambda UT^{1/2})a\rangle=\|a\|^2.\end{equation}
Since $U[C^*(T)T^{1/2}a]\perp[C^*(T)a]$, $U=U^*$ and $UT=TU$, we have in particular $$\langle (I+|\lambda|^2T)^{-1}a,UT^{1/2}a\rangle=0=\langle(I+|\lambda|^2T)^{-1}a,T^{1/2}Ua\rangle\ \ \ (\mbox{when}\ \lambda\ne0),$$ hence,  the left side of (\ref{51}) is equal to
$$\langle(I+|\lambda|^2T)^{-1}a,(I+\overline{\lambda}T^{1/2}U)(I+\lambda UT^{1/2})a\rangle=\langle(I+|\lambda|^2T)^{-1}a,(I+|\lambda|^2T)a\rangle=\|a\|^2.$$
\end{proof}

Now we would like to reformulate (\ref{15}) and (\ref{16}) so that $a$ and $b$ would appear  symmetrically.
Using the definition of $T$, it can easily be proved by and induction that, assuming $C^*=C$, (\ref{15}) is equivalent to
\begin{equation}\label{25}\langle(C^{2n}a,Cb\rangle=0\ \ \ (n=0,1,2,\ldots).\end{equation}

{\bf Suppose now that $\dim\h<\infty$.} Then it follows from Proposition \ref{pr2} that $\det C=0$ (hence $0$ must be an eigenvalue of $C$) since $\det C$ is the coefficient of $\lambda$ in the development of $\det(A+\lambda B)$. Observe that we need to verify (\ref{15}) and (\ref{16}) only for $n$ smaller than the degree $m$ of the minimal polynomial of $T$, since $T^m, T^{m+1},\ldots$ can all be expressed as linear combinations of $T^k$ for $k<m$. Let $\gamma_1,\ldots\gamma_m$ be the nonzero eigenvalues of $C$.  Let $a_j$ and $b_j$ be the components of $a$ and $b$ in the eigenspace $\ker(C-\gamma_jI)$. Then (\ref{25}) can be written as
$$\sum_{j=1}^m\gamma_j^{2n+1}\langle a_j,b_j\rangle=0\ \ (n=0,1,2,\ldots).$$  Since $\det[\gamma_j^{2n+1}]\ne0$ ($j=1,\ldots,m,\ n=0,\ldots, m-1$),  it follows that (\ref{25}) is equivalent to
\begin{equation}\label{40}\langle a_j,b_j\rangle=0\ \ (j=1,\ldots,m).\end{equation}

Note that $Ta=a-\langle a,a\rangle a-C^2a=\|b^2\|a-C^2a$, hence in the case $n=0$ the identity (\ref{16}) says that $\|Cb\|^2=\langle Ta,a\rangle=\|b\|^2\|a\|^2-\|Ca\|^2$, so that 
\begin{equation}\label{18}\|Ca\|^2+\|Cb\|^2=\|a\|^2\|b\|^2.\end{equation}
Let us now consider the case $n=1$ of (\ref{16}). Using (\ref{25}) and that $C^*=C$ we compute $TCb=(I-aa^*-C^2)Cb=Cb-C^3b$. Hence $\langle TCb,Cb\rangle=\|Cb\|^2-\|C^2b\|^2$ and therefore by (\ref{16}) in the case $n=1$ and using (\ref{25}) again we have
$$\|Cb\|^2-\|C^2b\|^2=\langle TCb,Cb\rangle=\langle T^2a,a\rangle
$$$$=\|Ta\|^2=\|\|b\|^2a-C^2a\|^2=\|b\|^4\|a\|^2-2\|b\|^2\|Ca\|^2+\|C^2a\|^2.$$
This can be rewritten  as
$$\|C^2a\|^2+\|C^2b\|^2=\|Cb\|^2-\|b\|^2(\|a\|^2\|b\|^2-2\|Ca\|^2).$$
By using (\ref{18}) and the identity $\|a\|^2+\|b\|^2=1$ the right side  simplifies to $\|Cb\|^2-\|b\|^2(\|Cb\|^2-\|Ca\|^2)=\|a\|^2\|Cb\|^2+\|b^2\|Ca\|^2$, hence
\begin{equation}\label{20}\|C^2a\|^2+\|C^2b\|^2=\|a\|^2\|Cb\|^2+\|b\|^2\|Ca\|^2.\end{equation}
If $n\geq3$, the computation, required to rewrite  (\ref{16})  in a way in which $a$ and $b$ appear symmetrically, seems to be so long  that the author is not able to accomplish it. In the example below we will need only the cases $n=0,1$. 

\begin{ex}Let us determine (up to equivalence) all $A\in\mat{3}{\bc}$ that satisfy $A\po B$, where $B\in\mat{3}{\bc}$ is a projection of rank $1$. By what we have established above we may suppose that $A$ and $B$ are of the form
$$\left[\begin{array}{ccc}
0&\overline{a}_1&\overline{a}_2\\
b_1&\alpha&0\\
b_2&0&0\end{array}\right],\ \ B=\left[\begin{array}{ccc}
1&0&0\\
0&0&0\\
0&0&0\end{array}\right],$$
where $0\leq\alpha\leq1$,
\begin{equation}\label{41}|a_1|^2+|a_2|^2+|b_1|^2+|b_2|^2=\|a\|^2+\|b\|^2=1\end{equation}
and from (\ref{40}), (\ref{18}) and (\ref{20}) 
\begin{equation}\label{42}\alpha a_1\overline{b_1}=0,\end{equation}
\begin{equation}\label{43}\alpha^2 (|a_1|^2+|b_1|^2)=(|a_1|^2+|a_2|^2)(|b_1|^2+|b_2|^2),\end{equation}
\begin{equation}\label{44}\alpha^4(|a_1|^2+|b_1|^2)=\alpha^2[(|a_1|^2+|a_2|^2)|b_1|^2+(|b_1|^2+|b_2|^2)|a_1|^2].\end{equation}
It is not hard to solve this system of equations to obtain for $A$ the matrices of the following forms or their transposes:
$$A=\left[\begin{array}{ccc}
0&0&0\\
b_1&0&0\\
b_2&0&0\end{array}\right],$$
where $|b_1|^2+|b_2|^2=1$, and
$$ A=\left[\begin{array}{ccc}
0&\overline{a}_1&0\\
0&\alpha&0\\
b_2&0&0\end{array}\right],$$
where   $|a_1|^2+|b_2|^2=1$, $0<\alpha\leq1$ and $|b_2|=\alpha$ or $|b_2|=1$.
\end{ex}

\end{document}